\documentclass{lms}

\usepackage{amssymb, mathrsfs, amsmath}

\newtheorem{theorem}{Theorem}[section] 
\newtheorem{lemma}[theorem]{Lemma}     

\newnumbered{assertion}{Assertion}    
\newnumbered{conjecture}{Conjecture}  
\newnumbered{definition}{Definition}
\newnumbered{hypothesis}{Hypothesis}
\newnumbered{remark}{Remark}
\newnumbered{note}{Note}
\newnumbered{observation}{Observation}
\newnumbered{problem}{Problem}
\newnumbered{question}{Question}
\newnumbered{algorithm}{Algorithm}
\newnumbered{example}{Example}
\newunnumbered{notation}{Notation} 

\def\dA{d_{\mathrm A}}
\def\d{{\rm d}}

\def\R{{\mathbb R}}
\def\N{{\mathbb N}}

\def\B{{\mathscr B}}

\def\<{\left\langle}
\def\>{\right\rangle}

\def\qqand{\qquad\mbox{and}\qquad}
\def\qand{\qand\mbox{and}\quad}

\def\tfrac#1#2{{\textstyle\frac{#1}{#2}}}

\def\be#1{\begin{equation}\label{#1}}
\def\ee{\end{equation}}

\title[Log-Lipschitz embeddings]%
 {Log-Lipschitz embeddings of homogeneous sets with sharp logarithmic exponents and slicing the unit cube} 

\author{James C.\ Robinson}

\classno{30L05, 52A20, 52A40, 54F45 (primary), 37L30 (secondary)}

\extraline{This work was supported by an EPSRC Leadership Fellowship, grant EP/G007470/1}

\begin{document}
\maketitle

\begin{abstract}
If $X$ is a subset of a Banach space with $X-X$
homogeneous, then $X$ can be embedded into some $\R^n$ (with $n$ sufficiently large) using a linear
map $L$ whose inverse is Lipschitz to within logarithmic corrections.
More precisely,
$$
c\,\frac{\|x-y\|}{|\,\log\|x-y\|\,|^\alpha}\le|Lx-Ly|\le c\|x-y\|
$$
for all $x,y\in X$ with $\|x-y\|<\delta$ for some $\delta$ sufficiently small.
A simple argument shows that one must have $\alpha>1$ in the case of a general Banach space and $\alpha>1/2$ in the case of a Hilbert space. It is
shown in this paper that these exponents can be achieved.

While the argument in a general Banach space is relatively straightforward, the Hilbert space case relies on a result due to Ball ({\em Proc. Amer. Math. Soc.} 97 (1986) 465--473) which guarantees that the maximum volume of hyperplane slices of the unit cube in $\R^d$ is $\sqrt2$, in dependent of $d$.
\end{abstract}

\section{Introduction} 
\label{intro}

\noindent The abstract results in this paper are motivated by a problem from the theory of infinite-dimensional dynamical systems, but are relevant to the theory of convex bodies and have consequences for the bi-Lipschitz embedding problem for compact metric spaces.

 Suppose that $X$ is a finite-dimensional subset of a Banach space $\B$ that is invariant under the (semi-)flow generated by the differential equation
$$
\dot u={\mathcal G}(u),\qquad u\in\B
$$
(many partial differential equations can be recast in this form by identifying an appropriate phase space in which the solutions evolve, see, for example, Temam, 1988, or Robinson, 2001). It is natural to ask whether one can construct a finite-dimensional set of ordinary differential equations that reproduce the dynamics on the set $X$ (cf.\ Eden et al., 1994; Robinson, 1999; Romanov, 2000).

Suppose for simplicity that $\mathcal G$ is Lipschitz from $X$ into itself (this is an unrealistic assumption in general, but one can obtain information about the smoothness of $\mathcal G$ in certain particular cases, see Pinto de Moura \& Robinson, 2010b, for example). A `straightforward' way to try to construct such a finite-dimensional system is to find an embedding of $X$ into some $\R^N$, i.e.\ a mapping $L:\B\rightarrow\R^N$ that is one-to-one between $X$ and its image. In this case, the vector field on $LX$ that reproduces the dynamics on $X$ is
$$\hat g(x)=[L^{-1}\circ{\mathcal G}\circ L](x),$$
and this can be extended to a vector field $g$ defined on the whole of $\R^N$ using any extension result that preserves the modulus of continuity (e.g.\ McShane, 1934; Stein, 1970; Wells \& Williams, 1975). However, it remains to guarantee that the solutions of the finite-dimensional system $\dot x=g(x)$ are unique; in general this is assured provided that $g$ is $1$-log-Lipschitz, i.e.\ there exist $c>0$ and $\delta>0$ such that
$$
|g(x)-g(y)|\le c|x-y|\bigl|\log|x-y|\bigr|\qquad\mbox{for all}\qquad x,y\ \mbox{with}\ |x-y|<\delta
$$
(that this is sufficient for uniqueness follows from Osgood's criterion, $\int_0^1 \omega^{-1}(s)\,\d s=\infty$, where $\omega$ is the modulus of continuity of $g$; see Hartman, 1964, for example).

If one takes $L$ to be linear, the modulus of continuity of $g$ is determined by the modulus of continuity of $L^{-1}$. So the reproduction of the `finite-dimensional dynamics' on $X$ within a finite-dimensional system of ODEs that have unique solutions relies on finding an embedding of $X$ into $\R^N$ whose inverse is log-Lipschitz with a sufficiently small logarithmic exponent.

Study of the regularity of embeddings of general finite-dimensional sets into Euclidean spaces (where `regularity' means regularity of the inverse mapping) began with a result for subsets of Euclidean spaces with finite box-counting dimension, due to Ben-Artzi et al.\ (1993), showing that in this case $P^{-1}$ is H\"older continuous for `most' orthogonal projections $P$ onto a space of sufficiently high dimension; Foias \& Olson (1996) then showed that the same is true for subsets of infinite-dimensional Hilbert spaces (with finite box-counting dimension), and Hunt \& Kaloshin (1999) gave a sharp bound on this H\"older exponent, treating linear maps rather than projections and providing the essential steps for extending the argument to subsets of Banach spaces\footnote{Some additional work is in fact required for their techniques to be applicable in the Banach space case, see Robinson (2009).}. However, the fact that in general one can do no better than H\"older continuous for $L^{-1}$ means that in the context of the `dimension reduction' programme discussed above, the result will be a H\"older continuous ordinary differential equation on $\R^N$, for which no uniqueness can be guaranteed.

It is therefore natural to turn to other more restrictive definitions of dimension, to see if these can be exploited to improve the modulus of continuity of $L^{-1}$. This paper concentrates on the Assouad dimension, which was introduced in the context of metric spaces
(Assouad, 1983; see also Bouligand, 1928), and has been used extensively in the search for conditions under which an arbitrary metric space admits a bi-Lipschitz embedding into some $\R^k$ (see Heinonen, 2003, for more on this problem, and Luukaainen, 1998, or Olson, 2002, for more on the Assouad dimension).

This dimension is most naturally
defined as a concept auxiliary to the notion of a homogeneous
space:

\begin{definition}
A metric space $(X,d)$ is said to be $(M,s)$-{\it homogeneous\/} (or
simply {\it homogeneous\/}) if any ball of radius~$r$ can be covered
by at most $M(r/\rho)^s$ smaller balls of radius~$\rho$. The \emph{Assouad dimension} of $X$, $\dA(X)$, is the infimum of all
$s$ such that $(X,d)$ is $(M,s)$-homogeneous for some $M\ge 1$.
\end{definition}

Since any subset of $\R^N$ is homogeneous and homogeneity is
preserved under bi-Lipschitz mappings, it follows that $(X,d)$ must
be homogeneous if it is to admit a bi-Lipschitz embedding into some
$\R^N$. However, the example of the Heisenberg group with the Carnot-Carath\'eodory metric shows that homogeneity is not sufficient to guarantee a bi-Lipschitz embedding (see Semmes, 1996; there are other counterexamples due to Laakso, 2002).

However, Olson (2002) showed that if $X$ is a subset of $\R^N$ with $\dA(X-X)<s$, there is an `almost bi-Lipschitz' embedding of $X$ into $\R^k$, i.e.\ an embedding that is linear and has a log-Lipschitz inverse; here,
$$
X-X=\{x-y:\ x,y\in X\}
$$
(note that $X-X$ contains a (perhaps translated) copy of $X$, so clearly $\dA(X-X)\ge\dA(X)$). This result was extended by Olson \& Robinson (2010) for subsets of real Hilbert spaces, and their argument was subsequently adapted by Robinson (2009) to treat subsets of real Banach spaces, yielding the following result which forms the focus of this paper. The notion of `prevalence', used in the theorem, is briefly recalled in Section \ref{prev}.

\begin{theorem}\label{thm:dA}
Let $X$ be a compact subset of a Banach space $\B$ with $d_{\rm
A}(X-X)<d<N$, where $N\in\N$.  If
\begin{equation}\label{eq:logexp}
\gamma>\frac{\alpha N+1}{N-d},
\end{equation}
where $\alpha=2$ if $\B$ is a Banach space and $\alpha=\tfrac{3}{2}$ if $\B$ is in fact a Hilbert space,
then a prevalent set of linear maps $L:\B\rightarrow\R^N$ is $\gamma$-almost bi-Lipschitz: for some constants $c_L>0$, $\rho_L>0$,
\begin{equation}\label{gammabi}
\frac{1}{c_L}\,\frac{\|x-y\|}{|\,\log\|x-y\|\,|^\gamma} \le
\|Lx-Ly\|\le c_L\|x-y\|
\end{equation}
for all $x,y\in X$ with $\|x-y\|\le\rho_L$.
\end{theorem}

Given this result, it is natural to ask whether or not the exponent $\gamma$ is sharp; here this is taken to mean whether the limiting value as $N\rightarrow\infty$, namely $\gamma>\alpha$, is optimal. Following an approach developed by Ben-Artzi et al.\ (1999), Pinto de Moura \& Robinson (2010a) showed that for the simple example of an orthogonal set whose norm decays exponentially one cannot improve on $\gamma>1$ in the Banach space case and $\gamma>\tfrac{1}{2}$ in the Hilbert space case. This paper closes the gap between these limits and the exponents in the statement of the theorem above, and shows that one can indeed take $\alpha=1$ in the Banach space case and $\alpha=\tfrac{1}{2}$ in the Hilbert space case.

Since one can isometrically embed any compact metric space $(X,d)$ into the Banach space $L^\infty(X)$ using the Kuratowski mapping $x\mapsto d(\cdot,x)$, the condition in Theorem \ref{thm:dA} for subsets of Banach spaces can be translated into a theorem for compact metric spaces, where one has to interpret the condition `$X-X$ is homogeneous' in terms of the image of $X$ in $L^\infty$ under the Kuratowski mapping. It would be interesting to obtain a more intrinsic characterisation of such metric spaces.

While the argument to obtain the optimal exponents in a general Banach space is relatively straightforward, the Hilbert space case relies on the following result due to Ball (1986) which guarantees that the maximum volume of hyperplane slices of the unit cube is $\sqrt2$, independent of the dimension. (Hensley, 1979, had previously showed that this upper bound is $\le 5$; Ball's argument is very similar, but he takes more care to refine the upper bound.)

\begin{theorem}[(Ball, 1986)]\label{ballthm}
Let $I_n=[-\tfrac{1}{2},\tfrac{1}{2}]^n\subset\R^n$ be the unit cube in $\R^n$, and let $S$ be a
co-dimension $1$ subspace of $\R^n$ with unit normal $a$. Then for any $r\in\R$
$$
|(S+ra)\cap I_n|\le\sqrt2.
$$
\end{theorem}

In terms of the motivation discussed above, despite the reduction of the logarithmic exponent in the Hilbert space case to any $\gamma>1/2$, there are two outstanding problems: first, the smoothness of the original vector field $\mathcal G$ on $X$, which is only known to be $1$-log-Lipschitz in general (Pinto de Moura \& Robinson, 2010b), rather than Lipschitz; and more tellingly, the lack of any techniques for bounding the Assouad dimension of invariant sets of infinite-dimensional dynamical systems (the strongest current results, see Pinto de Moura, Robinson, \& S\'anchez-Gabites, 2010, simply assume that $\dA(X-X)$ is finite).

\section{`Prevalence' and some preparatory lemmas}\label{prev}

The main theorem of this paper shows that there is a `prevalent' set of linear mappings from $\B$ into $\R^k$ with log-Lipschitz inverses. Prevalence provides an infinite-dimensional generalisation of the idea of `almost every'; it was in fact first used by Christensen (1973) in the study of the differentiability of Lipschitz mappings between infinite-dimensional spaces, a similar theory being developed later (but independently) by Hunt et al.\ (1992), mainly to deal with problems in the field of dynamical systems (see also Ott \& Yorke, 2005).

More formally, a subset $S$ of a normed linear space $V$ is prevalent if there exists a compactly supported probability measure $\mu$ on $V$ such that $\mu(v+S')=1$ for every $v\in V$, where $S'$ is a Borel set contained in $S$. The support of $\mu$ provides a `probe set' $E$ of allowable perturbations: given $v\in V$, $S$ is prevalent if $v+e\in S$ for $\mu$-almost every $e\in E$.
A major step in proving that any set $S$ is prevalent is the construction of an appropriate probability measure $\mu$ (and thus of $E$).

The construction of $E$ here, tailored to a particular set $X$, relies on the following lemma, whose proof can be found in Robinson (2009). [The statement of the lemma there is slightly different, but the proof of the result as stated here is identical.]

\begin{lemma}\label{Ujs}
Suppose that $Z$ is a compact homogeneous subset of $\B$. Then there exists an $M'>0$ and a sequence of
linear subspaces $\{V_j\}_{j=0}^\infty$ of $\B^*$ with $\dim V_j\le
M'$ for every $j$, such that for any $z\in Z$ with
$2^{-(n+1)}\le\|z\|\le 2^{-n}$, there exists an element $\psi\in
V_n$ such that
$$
\|\psi\|=1\qquad\mbox{and}\qquad|\psi(z)|\ge 2^{-(n+3)}.
$$
\end{lemma}

%
%
%
%
%

In a Hilbert space it is more helpful to use the following result; note that the spaces
$V_j$ are now mutually orthogonal, but that the space $V_n$ alone is not sufficiently `rich' to obtain (\ref{Pinz}).

\begin{lemma}\label{Vjs}
Suppose that $Z$ is a compact homogeneous subset of a Hilbert space
$H$. Then there exists an $M'>0$ and a
sequence $\{V_j\}_{j=0}^\infty$ of mutually orthogonal linear subspaces
 of $H$, with $\dim V_j\le M'$ for every
$j$, such that for any $z\in Z$ with $2^{-(n+1)}\le\|z\|\le 2^{-n}$,
\be{Pinz}
\|\Pi_nz\|\ge 2^{-(n+2)},
\ee
where $\Pi_n$ is the orthogonal projection onto $\oplus_{j=1}^nV_j$.
\end{lemma}

\begin{proof}
Write
$$
Z_j=\{z\in Z:2^{-(j+1)}\le \|z\|\le 2^{-j}\}.
$$
Since $Z_j\subset B(0,2^{-j})$ it can be covered by $M_j$ balls
of radius $2^{-(j+2)}$, with centres $\{u^{(j)}_i\}_{i=1}^{M_j}$, where
$$ M_j=N(2^{-j},2^{-(j+2)})\le
    4^sM=M'.
$$
Let $U_j$ be the space spanned by $\{u_i^{(j)}\}_{i=1}^{M_j}$; clearly ${\rm dim}(U_j)\le M'$, and if $P_j$ denotes the projection onto $U_j$,
$$
\|P_jz\|\ge\|z\|-\|z-P_jz\|\ge 2^{-(n+1)}-2^{-(n+2)}=2^{-(n+2)}.
$$
Finally, define mutually orthogonal subspaces $V_j$ such that
$$
\bigoplus_{j=1}^n V_j=\bigoplus_{j=1}^n U_j.
$$
and the result follows since $\|\Pi_nz\|\ge\|P_nz\|$.
\end{proof}

The spaces whose existence is guaranteed by these two Lemmas with $Z=X-X$ form the basis of the construction of the `probe space' with respect to which it will be shown that linear embeddings with log-Lipschitz inverses are prevalent.

In the case of a general Banach space, take $S_j$ to be the unit ball in $V_j$ (the spaces constructed in Lemma \ref{Ujs}, which, note, are subsets of $\B^*$), choose $s>1$, and let $E$ be the collection of all maps $L:\B\rightarrow\R^N$ given by
\begin{equation}\label{Egammadef}
E=\left\{L=(L_1,\ldots,L_N):\ L_n=\sum_{j=1}^\infty
j^{-s}\phi_{nj},\ \phi_{nj}\in S_j\right\}.
\end{equation}
The choice $s>1$ guarantees (via the triangle inequality) that the expression for each $L_n$ converges. To define a measure on $E$, the first step is to define a
measure on $S_j$ for each $j$. To do this, first choose a basis for
$V_j$; then by means of the coordinate representation with respect
to this basis one can identify $S_j$ with a symmetric convex set
$U_j\subset\R^{d_j}$. The uniform probably measure on
$U_j$ induces a probability measure $\lambda_j$ on $S_j$. Finally,
define the measure $\mu$ on $E$ to be
that in which each $\phi_{nj}$ is chosen independently and at random
according the the distribution $\lambda_j$.

If $\B$ is in fact a Hilbert space, instead of a `unit ball' $S_j$ take a `unit cube' $C_j$ in $V_j$: choose an orthonormal basis $\{e_j^{(i)}\}_{i=1}^{d_j}$ for $V_j$, and let
\be{ucube}
C_j=\{u\in V_j:|(u,e_j)|\le\tfrac{1}{2}\}.
\ee
  Let $s>1/2$, and let $E$ be the collection of all maps $L:\B\rightarrow\R^k$ given by
\begin{equation}\label{Egammadef}
E=\left\{L=(L_1,\ldots,L_N):\ L_n=\left[\sum_{j=1}^\infty
j^{-s}\phi_{nj}\right]^*,\ \phi_{nj}\in C_j\right\},
\end{equation}
where for $x\in H$, $x^*$ denotes the linear map $u\mapsto (x,u)$.
Since the spaces $V_j$ are mutually orthogonal,
$$
\left\|\sum_{j=1}^\infty
j^{-s}\phi_{nj}\right\|^2=\sum_{j=1}^\infty
j^{-2s}\|\phi_{nj}\|^2\le\sum_{j=1}^\infty j^{-2s},
$$
and the condition $s>1/2$ is now sufficient to ensure that the expression for
$L_n$ converges. We can define a measure $\mu$ on $E$ simply by letting $\phi_{nj}$ be uniformly distributed over $C_j\simeq I_{d_j}$, where $I_d=[-\tfrac{1}{2},\tfrac{1}{2}]^d$ is the unit cube in $\R^d$.\medskip

The following bound is central to the proof, and key to the improvement in the exponent of the logarithmic term. In the Banach space case it follows from an argument due to Hunt \& Kaloshin (1999; Lemma 3.10); in the Hilbert space case the argument is much more delicate (since the orthogonal spaces used in the construction of $E$ are `smaller' than the spaces used in the Banach space case), and relies on Ball's result about hyperplane slices of products of the unit cube. In the statement of the lemma, ${\mathscr L}(\B,\R^N)$ denotes the space of all bounded linear maps from $\B$ into $\R^N$.

\begin{lemma}\label{lem:mu_bound}If $z\in Z$ with $2^{-(j+1)}\le\|z\|\le 2^{-j}$ then for any $f\in{\mathscr L}(\B,\R^N)$,
\begin{equation}\label{eq:key}
    \mu\{\, L\in E :\ |(f+L)z|<\epsilon 2^{-j}\,\}
    \le C\epsilon^N j^{s N},
\end{equation}
where $C=C(N)$.
\end{lemma}

\begin{proof}
  In the Banach space case, Hunt \& Kaloshin show that for any $z\in\B$ and any $\psi\in S_j$ with $\|\psi\|_*=1$,
  $$
  \mu\{L\in E:\ |(f+L)z|<\epsilon\}\le \left(j^sd_j\epsilon|\psi(z)|^{-1}\right)^N.
  $$
  (For the result in precisely this form see Lemma 6.10 in Robinson, 2010.)
  In the case considered here, $d_j={\rm dim}(V_j)\le M'$, and there exists a $\psi\in S_j$ with $\|\psi\|=1$ such that $\psi(z)\ge 2^{-(j+3)}$ using Lemma \ref{Ujs}, from which (\ref{eq:key}) follows immediately. 
  
  In the Hilbert space case, observe that
\begin{align*}
&\mu\{L\in E:\ |(f+L)(x)|<\epsilon\}\\
&\quad\le\mu\{L=(l_1,\ldots,l_k)\in E:\ |(f_n+l_n)(x)|<\epsilon\ \mbox{for each }n=1,\ldots,k\}\\
&\quad=\prod_{n=1}^k\mu_0\{l\in E_0:\ |(f_n+l)(x)|<\epsilon\}.
\end{align*}
Take an $f_0\in H^*$ and consider
\begin{align*}
&\left[\bigotimes_{i=1}^\infty\lambda_i\right]\left\{\{\phi_i\}_{i=1}^\infty\in E_0:\ |f_0(x)+\sum_{i=1}^\infty i^{-\gamma}(\phi_i,x)|<\epsilon\right\}\\
&\qquad=\left[\bigotimes_{i=1}^\infty\lambda_i\right]\left\{\{\phi_i\}_{i=1}^\infty\in E_0:\ |[f_0(x)+\sum_{i\neq j}^\infty i^{-\gamma}(\phi_i,x)]+j^{-\gamma}(\phi_j,x)|<\epsilon\right\}.
\end{align*}
 It will soon be shown that for $\alpha=f_0(x)+\sum_{i\ge n+1}^\infty i^{-\gamma}(\phi_i,x)$ fixed, the bound on
$$
\lambda_j\{\phi\in S_j:\ |\alpha+\sum_{j=1}^nj^{-\gamma}(\phi_j,x)|<\epsilon\}
$$
is independent of $\alpha$. It follows from the product structure of the measure $\otimes_{j=1}^\infty\lambda_j$ that the above expression is bounded by 
$$
\left[\bigotimes_{j=1}^n\lambda_j\right]\{(\phi_1,\ldots,\phi_n)\in\prod_{j=1}^nC_j:\ |\sum_{j=1}^nj^{-\gamma}\phi_j^*(\Pi_jx)|<\epsilon\}.
$$
The estimate now depends on an entirely finite-dimensional problem. Indeed, each $V_j\simeq\R^{d_j}$, and $C_j$ (the `unit cube' in $V_j$) is isomorphic to $I_{d_j}$. Set $D=\sum_{j=1}^nd_j$. The vector $(P_1x,\ldots,P_nx)$ corresponds to a vector $a=(a_1,\ldots,a_n)\in\R^D$; if we set
$$
a'=(a_1,2^{-s}a_2,\cdots,n^{-s}a_n)\qqand \hat a=a'/|a'|
$$
and let $\mu$ denote the uniform probability measure on $I_D$ (i.e.\ Lebesgue measure), the problem is to bound, for any $y\in\R$,
\begin{align*}
\mu\{x\in I_D:\ |y+(x\cdot a')|\le\epsilon\}&=\frac{1}{|a'|}\,\mu\{x\in I_D:\
|y+(x\cdot\hat a)|\le\epsilon\}\\
&\le \frac{n^s}{|a|}\,\mu\{x\in I_D:\
|y+(x\cdot\hat a)|\le\epsilon\}.
\end{align*}
where $\hat a=a'/|a'|$. The result is now a consequence of Theorem \ref{ballthm}, since
$$
\mu\{x\in I_D:\ |y+(x\cdot\hat a)|\le\epsilon\}\le 2\epsilon|(S_{\hat
a}-y\hat a)\cap I_D|\le 2\epsilon\sqrt 2
$$
where $S_{\hat a}$ is the hyperplane through the origin with normal $\hat a$.\end{proof}

\section{Log-Lipschitz embeddings with sharp exponent}

Armed with Lemma \ref{lem:mu_bound} the argument that gives the sharp exponents is relatively straightforward. The short proof, reproduced here in order to make this paper self-contained, is taken from Robinson (2009).

\begin{theorem}\label{thm:dAsharp}
In Theorem \ref{thm:dA} one can take $\alpha=1$ if $\B$ is a Banach space and $\alpha=\tfrac{1}{2}$ if $\B$ is in fact a Hilbert space.
\end{theorem}

\begin{proof}
Choose $s>\alpha$ small enough to
ensure that
\begin{equation}\label{zetachoose}
\gamma>\frac{sN+1}{N-d}
\end{equation}
and define the probe set $E$ following the construction outlined above using this value of $s$.

Define a sequence of layers of
$X-X$,
\begin{equation}\label{Zlayers}
Z_j=\{z\in X-X:\ 2^{-(j+1)}\le\|z\|\le 2^{-j}\}
\end{equation}
and for a given $f\in{\mathscr L}(\B,\R^N)$, let $Q_j$ be the corresponding set of maps that fail to satisfy the almost
bi-Lipschitz property for some $z\in Z_j$,
$$
Q_{j}=\{\,L\in Q:\ |(f+L)z|
        \le j^{-\gamma}2^{-j}
    \ \mbox{ for some }z\in Z_j\,\}.
$$
For every $L\in Q$, the map $f+L$ is Lipschitz, with Lipschitz constant no larger than $K$.

By assumption $\dA(X-X)<d$, and so $Z_j\subset B(0,2^{-j})$ can be
covered by $M_j\le Mj^{\gamma d}$ balls of radius
$j^{-\gamma}2^{-j}$, whose centres $z_i^{(j)}$ ($i=1,\ldots,M_j$) lie in $Z_j$. Given any $z\in Z_j$ there is
$z^{(j)}_i$ such that $\|z-z^{(j)}_i\|\le j^{-\gamma}2^{-(j+2)}$, and thus
\begin{align*}
    |(f+L)z|&\ge |(f+L)z^{(j)}_i|
            -|(f+L)(z-z^{(j)}_i)|\\
        &\ge |(f+L)(z^{(j)}_i)|
            -Kj^{-\gamma}2^{-j}
\end{align*}
which implies, using
%
Lemma \ref{lem:mu_bound}, that
\begin{align*}
    \mu(Q_{j})&\le \sum_{i=1}^{M_j}
        \mu \{\,L\in Q:\ |(f+L)z_i^{(j)}|
        \le (1+K)j^{-\gamma}2^{-j}\,\}\\
        &\le M_j C(1+K)^Nj^{-\gamma N}j^{s N}\\
        &\le C'j^{\gamma d-N(\gamma-s)},
\end{align*}
since $M_j\le Mj^{\gamma d}$. The condition
(\ref{zetachoose}) implies that $\gamma d+N(s-\gamma)<-1$, and so
$$\sum_{j=1}^\infty\mu(Q_{j})< \infty.$$
Using the Borel--Cantelli Lemma, $\mu$-almost every $L$ is contained
in only a finite number of the $Q_j$: thus for $\mu$-almost every
$L$ there exists a $j_L$ such that for all $j\ge j_L$,
$$
2^{-(j+1)}\le\|z\|\le2^{-j}\quad\Rightarrow\quad |Lz|\ge
j^{-\gamma}2^{-j},
$$
so for $\|z\|\le 2^{-j_L}$,
\begin{equation}\label{almostconclusion}
|Lz|\ge 2^{-(1+\gamma)}\frac{\|z\|}{|\,\log\|z\|\,|^\gamma}.
\end{equation}
\end{proof}

%
%

\begin{acknowledgements}\label{ackref}
I would particularly like to Eleonora Pinto de Moura for many interesting conversations about this problem, and Eric Olson with whom the groundwork was laid.
\end{acknowledgements}

\affiliationone{
   James C.\ Robinson\\
   Mathematics Institute, University of Warwick, Coventry CV4 7AL.\\
   U.K.
   \email{j.c.robinson@warwick.ac.uk}}%
\end{document}